\documentclass[11pt,twoside]{article}

\title{Local holomorphic Euler characteristic and instanton decay}
\author{E.\ Gasparim \and T.\ K\"{o}ppe \and P.\ Majumdar}
\date{}

\usepackage[pdftitle={Local holomorphic Euler characteristic and instanton decay},
            pdfauthor={Elizabeth Gasparim, Thomas Köppe, Pushan Majumdar},
            pdfsubject={Mathematics, Algebraic Geometry, Mathematical Physics}, pdfproducer={}, pdfcreator={},
            pdfpagelayout=TwoColumnRight, colorlinks=true, linkcolor=black, urlcolor=black, citecolor=black]{hyperref}
\usepackage{amsmath,amsthm,amssymb,amscd,bm,fancyhdr}
\usepackage[outer=3cm,inner=3cm]{geometry}

\numberwithin{equation}{section}
\newtheorem{theorem}[equation]{Theorem}
\newtheorem{proposition}[equation]{Proposition}
\newtheorem{corollary}[equation]{Corollary}

\newtheorem{lemma}[equation]{Lemma}
\theoremstyle{remark}
\newtheorem{remark}[equation]{Remark}

\theoremstyle{definition}
\newtheorem{definition}[equation]{Definition}
\newtheorem{example}[equation]{Example}

\newtheorem*{prop.mins}{Proposition \ref{prop.mins}}
\newtheorem*{prop.isomodk}{Proposition \ref{isomodk}}
\newtheorem*{prop.localKH}{Proposition \ref{localKH}}
\newtheorem*{corr.inst}{Corollary \ref{inst}}
\newtheorem*{thm.mini}{Theorem \ref{mini}}

\newcommand{\norm}[1]{\left\Vert#1\right\Vert}
\newcommand{\f}{\mathcal{F}}
\newcommand{\loc}{\mathrm{loc}}           
\newcommand{\hodge}{\raisebox{-.75mm}{*}} 
\newcommand{\ce}{\mathrel{\mathop:}=}     
\DeclareMathOperator{\Ext}{Ext}           
\DeclareMathOperator{\Sing}{Sing}         
\DeclareMathOperator{\Tot}{Tot}           
\DeclareMathOperator{\Elm}{Elm}           
\DeclareMathOperator{\coker}{coker}       
\DeclareMathOperator{\gc}{global\ charge}
\DeclareMathOperator{\lc}{local\ charge}

\begin{document}

\maketitle

\begin{abstract}\noindent
We study the local holomorphic Euler characteristic $\chi\bigl(x,
\f\bigr)$ of sheaves near a surface singularity obtained from
contracting a line $\ell$ inside a smooth surface $Z$. We prove
non-existence of sheaves with certain prescribed numerical
invariants. Non-existence of instantons on $Z$ with certain charges
follows, and we conclude that $\ell^2$ poses an obstruction to
instanton decay. A \emph{Macaulay~2} algorithm to compute $\chi$ is
made available at \url{http://www.maths.ed.ac.uk/~s0571100/Instanton/}.
\end{abstract}

\section{Introduction}

Let $\sigma \colon (\widetilde{X}, D) \to (X,x)$ be a resolution of an
isolated quotient singularity. Let $\widetilde\f$ be a reflexive sheaf
on $\widetilde{X}$, set $\f \ce (\sigma_*\widetilde\f)^{\vee\vee}$;
notice that there is an embedding $\sigma_*\widetilde\f \hookrightarrow \f$.
Then the \emph{local holomorphic Euler characteristic} of $\widetilde\f$
at $x$ is defined by
\begin{equation}\label{euler}
  \chi\bigl(x, \widetilde\f\bigr) \ce \chi\bigl((\widetilde{X}, D), \widetilde\f\bigr)
  \ce h^0\bigl(X; \ \f\bigl/\sigma_* \widetilde\f\bigr) + \sum_{i=1}^{n-1}(-1 )^{i-1}
     h^0\bigl(X; R^i \sigma_* \widetilde\f\bigr) \text{ .}
\end{equation}
For the case when $X$ is an orbifold, Blache \cite{Bl} shows that:
\begin{equation}\label{blache}
  \chi\bigl(\widetilde{X}, \widetilde\f\bigr) = \chi\bigl(X, \f\bigr) +
  \sum_{x \in \Sing X} \chi\bigl(x, \widetilde\f\bigr) \text{ .}
\end{equation}

In this paper we consider rational surface  singularities obtained by
contracting a line $\ell \cong \mathbb{P}^1$ with $\ell^2 <-1$ inside a
smooth surface. To calculate $\chi$ locally, it is enough to study
sheaves on a small neighbourhood of the singular point, or on a small
neighbourhood of the exceptional set of a resolution. We therefore
consider the spaces $Z_k \ce \Tot\bigl(\mathcal{O}_{\mathbb{P}^1}(-k)\bigr)$.

We denote by $X_k$ the space obtained from $Z_k$ by contracting the
zero-section $\ell$ to a point, and we let $\pi \colon Z_k \to X_k$ be
the contraction map. Since we are interested in applications to
instantons, we will consider sheaves $E$ over $Z_k$ with
$c_1(E) = 0$. Then $E\vert_{\ell}$ splits by Grothendieck's lemma, and
there exists an integer $j \geq 0$ called the \emph{splitting type} of
$E$ such that $E\vert_{\ell} \cong \mathcal{O}(j) \oplus \mathcal{O}(-j)$.
Set $Z_k^o \ce Z_k - \ell.$
We make two simple observations about reflexive sheaves on $Z_k$.

\begin{prop.mins}
Let $E$ be a rank-$2$ reflexive sheaf on $Z_k$ with splitting type
$\geq k$, then $\chi\bigl(x, E\bigr) \geq k - 1$.
\end{prop.mins}

\begin{prop.isomodk}
Let $E_1$ and $E_2$ be $\mathfrak{sl}(2, \mathbb{C})$-bundles over
$Z_k$ with splitting types $j_1$ and $j_2$, respectively. There exists
an isomorphism $E_1\vert_{Z_k^o} \cong E_2\vert_{Z_k^o}$ if and only
if $j_1 \equiv j_2 \mod k.$ In particular, $E_1$ can decay totally
over $Z_k$ if and only if $j_1 \equiv 0 \mod k$.
\end{prop.isomodk}

This paper consists of applications of the local holomorphic Euler
characteristic to problems of existence and decay of instantons. We
also discuss the Kobayashi--Hitchin correspondence over $Z_k$. We
obtain, via discussion of the physical consequences and an \textit{ad
hoc} definition of stability (Definition \ref{stable}), the following
conclusions:

\begin{prop.localKH}
There is a one-to-one correspondence between framed $SU(2)$-in\-stan\-tons
on $Z_k$ with local charge $n$ and framed-stable $\mathfrak{sl}(2,
\mathbb{C})$-bun\-dles on $Z_k$ with $\chi^{\loc} = n$.
\end{prop.localKH}

\begin{corr.inst}
An $\mathfrak{sl}(2, \mathbb{C})$-bundle over $Z_k$ represents an
instanton if and only if its splitting type is a multiple of $k$.
\end{corr.inst}

\begin{thm.mini}
The minimal local charge of a nontrivial $SU(2)$-instanton on $Z_k$ is
$\chi^{\mathrm{min}}_k = k-1$. The local moduli space of (unframed)
instantons on $Z_k$ having fixed local charge $\chi^{\mathrm{min}}_k$
has dimension $k - 2$.
\end{thm.mini}

\paragraph{Acknowledgements:} 
We are honoured to contribute to this volume commemorating Fedya's
$60^\text{th}$ birthday and are grateful to Ludmil Katzarkov for
giving us this opportunity.

The first author thanks the hospitality and generous support of the
Mathematics Institute of the Universit\"{a}t M\"{u}nster, where part
of this work was carried out. We thank M.\ Stillman for assistance
with \emph{Macaulay 2}.

\section{Elementary background on instantons}

Given a principal $SU(2)$-bundle $P \to X$ over a Riemannian
$4$-manifold $X$ with $c_2(P) = n > 0$, an \emph{$SU(2)$-instanton of
charge $n$ on $X$} is a connection $A$ on $P$ minimising the the
Yang--Mills functional
\[ S_{\mathrm{YM}}(A) \ce \int_X \norm{F_A}^2  \text{ ,} \]
where $F_A$ is the curvature of the connection $A$. The Yang--Mills
equations are the Euler--Lagrange equations corresponding to the
functional $S_{\mathrm{YM}}$.

Being non-linear and of second order, the Yang--Mills equations are
quite difficult to study. Luckily a linearisation can be obtained
easily as follows: The Yang--Mills equation of motion is $D(A) \wedge
F(A) = 0$. But since the Jacobi identity $D(A) \wedge \hodge F(A) = 0$
always holds, any $A$ satisfying $ \hodge F_A = \pm F_A$ 
solves the Yang--Mills equations of motion. A connection $A$ whose
curvature satisfies
\[ \hodge F_A = -F_A \text{ ,} \eqno{(\mathrm{ASD})} \]
is called \emph{anti-self-dual}. Hence, anti-self-dual connections
minimise the Yang--Mills functional. For this reason the ASD equations
may be seen as a ``linear version'' of the Yang--Mills
equations. Subsequently, from the mathematical point of view,
instantons have become synonymous to anti-self-dual connections.

Over a compact K\"{a}hler surface $X$, the Kobayashi--Hitchin
correspondence (see \cite{LT}) provides an interpretation of
irreducible $SU(2)$-instantons of charge~$n$ as stable holomorphic
$SL(2)$-bundles over $X$ with second Chern class $c_2 = n$:
\begin{align}\label{KH}
  \left\{ \parbox{5.5cm}{\centering irreducible $SU(2)$-instantons\\of charge $n$} \right\}
  &\overset{\text{K.--H.}}{\Longleftrightarrow}
  \left\{ \parbox{4cm}{\centering stable $SL(2)$-bundles\\with $c_2 = n$} \right\} \text{ ,}\nonumber \\
  \nabla = \bar\partial + \partial &\longleftrightarrow \bar\partial \text{ .}
\end{align}

In particular, when $X$ is a ruled surface, the informal
interpretation of instantons as weighted, point-like configurations of
concentrated charge has a precise interpretation in terms of jumping
lines: An instanton bundle on a ruled surface has a generic splitting
$\mathcal{O}(a) \oplus \mathcal{O}(-a)$, which is the same for all but
finitely many lines of the ruling, called the jumping lines. In this
interpretation, the weight of the points in the configuration
corresponds to the multiplicity of the jumps, and the topological
charge (=~the second Chern class) is given as the sum of all
multiplicities.

Multiplicities of jumps on lines having $\ell^2 = 0$, i.e.\ on trivial
families $\mathbb{C}P^1 \times \mathbb{C}$, are quite well understood
(cf.\ \cite{BHMM}, \cite{HM}). Here we study the behaviour of
instantons around lines with negative self-intersection $\ell^2 =
-k$. The case $k = 1$ corresponds to the blow-up of a surface at a
point and was studied in several papers, e.g.\ \cite{GO}, \cite{Ga1},
\cite{Kn}. Near a $-k$-line, with $-k < 0$, an instanton has two
independent local numerical invariants: the \emph{height} and the
\emph{width} (see Definition \ref{localcharge}), whose sum gives the
\emph{multiplicity} or \emph{local charge}. The label ``local charge''
comes from the translation into algebraic geometry via the
Kobayashi--Hitchin correspondence:

In the first case, let $\widetilde{X}$ be a compact complex surface
containing a $-1$-line, and let $\pi \colon \widetilde{X} \to X$ be
the blow-down of $\ell$ to a point $x_0 \in X$. If $E$ is a bundle on
$\widetilde{X}$, then the local second Chern class of $E$ near $\ell$
is by definition
\[ c_2^{\loc}\bigl(\ell, E\bigr) \ce c_2\bigl(E\bigr) -
   c_2\bigl((\pi_*E)^{\vee\vee}\bigr) \text{ .} \]
Thus, a local version of the Kobayashi--Hitchin correspondence justifies
the terminology
\begin{align*}
  \left\{ \parbox{3cm}{\centering local charge of an instanton} \right\}
  &\overset{\text{K.--H.}}{\Longleftrightarrow}
  \Bigl\{ \parbox{3.65cm}{\centering local $c_2$ of a bundle} \Bigr\} \text{ .}
\end{align*}
Still near the $-1$-line, an application of Hirzebruch--Riemann--Roch gives 
\begin{equation}\label{riemannroch}
  c_2^{\loc}\bigl(\ell, E\bigr) =  h^0\bigl(X; (\pi_*E)^{\vee\vee}
  \bigl/ \pi_*E\bigr) + h^0\bigl(X; R^1 \pi_*E\bigr) \text{ ,}
\end{equation}
which by \eqref{euler} is the local holomorphic Euler characteristic
of $E$ near $\ell$.

However, the analogue of \eqref{riemannroch} for a bundle near a
$-k$-line, where $k \geq 2$, is a more complicated issue, simply because
contracting such a line produces a singularity; and there exist at
least three nonequivalent definitions of Chern classes for singular
varieties. To avoid carrying this problem over to instantons, it is
more convenient to simply consider the local holomorphic Euler
characteristic, and set:

\begin{definition}\label{localcharge}
Let $E$ be an instanton bundle over a smooth surface $Z$ containing a
$-k$ line $\ell$ and let $\pi\colon Z \rightarrow X$ be map that
contracts of $\ell$ to a point. The \emph{local charge of $E$ around
$\ell$} is:
\[ \chi^{\loc}\bigl(E\bigr) = \chi\bigl(\ell, E\bigr)
   \ce h^0\bigl(X; (\pi_*E)^{\vee\vee} \bigl/ \pi_* E \bigr)
      + h^0\bigl(X; R^1 \pi_* E\bigr) \text{ .} \]
The right hand side contains two independent holomorphic invariants:

\noindent $\mathbf{w}(E) \ce h^0\bigl(X; (\pi_*E)^{\vee\vee} \bigl/
\pi_* E \bigr)$ is called the \emph{width} of the instanton and
measures how far the direct image is from being a bundle;

\noindent $\mathbf{h}(E) \ce h^0\bigl(X; R^1 \pi_* E\bigr)$ is called
the \emph{height} of the instanton and measures how far the bundle is
from being a split extension.
\end{definition}

From equation \eqref{blache}, $\chi\bigl(E\bigr) =
\chi^{\loc}\bigl(E\bigr) + \chi\bigl((\pi_*E)^{\vee\vee}\bigr)$, so we
can say that the local charge measures the loss of total charge
suffered by contracting $\ell$.

\section{Holomorphic surgery and instanton decay}

We first describe informally the ideas behind holomorphic surgery and
decay and then give the precise definitions. A \emph{decay} of an
instanton is a transformation that lowers the total charge; a
\emph{local decay around $\ell$} is a transformation that keeps the
instanton fixed outside $\ell$ but lowers the local charge near
$\ell$, and consequently lowers the global charge as well.

When an instanton is represented by a holomorphic bundle $E,$ then
holomorphic surgery provides a precise way to obtain local decay: If
the surface $Z$ contains a line $\ell$ and $N = N(\ell)$ is a small
neighbourhood of $\ell$ in the analytic topology, then lowering the
charge of $E$ around $\ell$ means to replace $E\vert_N$ by some
$E'\vert_N$ with smaller local charge, while keeping $E$ fixed on $Z -
\ell$. The outcome is a new holomorphic bundle which is isomorphic to
$E$ over $Z-\ell$, but which has smaller $c_2$.

\begin{remark}
Note that holomorphic surgery takes one instanton bundle to another
instanton bundle; that is, the surgery process keeps $c_1=0$, and
consequently differs from the more familiar process of elementary
transformations (which does not move between instantons).
\end{remark}

\begin{definition}
Two instanton bundles $E_1$ and $E_2$ on $X$ that are related by
holomorphic surgery around $\ell$ must satisfy $E_1\vert_{X - \ell}
\cong E_2\vert_{X - \ell}$ holomorphically.
 If in addition $c_2(E_1) > c_2(E_2)$, we will say that
\emph{$E_1$ decays to $E_2$}. 
\end{definition}

We fix a compact surface $Z$ containing a line $\ell$ with $\ell^2 =
-k$, as in the introduction, with $\pi \colon Z \to X$ the contraction
of $\ell$, and a decomposition $Z = (Z-\ell) \cup N(\ell)$. We will
see in Proposition \ref{kobhitchin} that instantons correspond to
bundles that are trivial on $Z^o \ce N(\ell) - \ell = (Z-\ell) \cap
N(\ell)$. So, instanton bundles can be given frames on $Z^o$. We will
use theses frames (see definition \ref{frame} below) in our
constructions. Note that $(Z_k, Z_k^o)$ describes the local situation,
whereas $(Z, Z^o)$ describes the global situation on a compact
manifold.
(Although it is not true that any $2$-dimensional tubular
neighbourhood of a $-k$-line $\ell$ is biholomorphic to $Z_k$, it is a
consequence of \cite{Ga2} that holomorphic bundles on both $N(\ell)$
and $Z_k$ are completely determined by a finite infinitesimal
neighbourhood of $\ell$, so that for the purposes of holomorphic
surgery and instanton decay, $N(\ell)$ and $Z_k$ can be identified.)

\begin{definition}\label{frame}
Let $\pi_F \colon F \to Z$ be a bundle over $Z$ that is trivial over
$Z^{o} \ce Z - Y$. Given two pairs $f =(f_1, f_2) \colon Z^{o} \to
\pi_{F}^{-1}(Z^{o})$ and $g =(g_1, g_2) \colon Z^{o} \to
\pi_{F}^{-1}(Z^{o})$ of fibrewise linearly independent holomorphic
sections of $F\vert_{Z^{o}}$, we say that $f$ is \emph{equivalent} to
$g$ (written $f \sim g$) if $\phi \ce g \circ f^{-1} \colon
V\rvert_{Z^o} \to V \vert_{Z^o}$ extends to a holomorphic map $\phi
\colon F \to F$ over the entire $Z$. A \emph{frame} of $F$ over $Z^o$
is an equivalence class of fibrewise linearly independent holomorphic
sections of $F$ over $Z^o$.



\begin{itemize}
\item A \emph{framed} bundle $\bar E^f$ on $Z$ is a pair consisting of
      a bundle $\pi_{\bar E} \colon \bar E \to Z$ together with
      a frame of $\bar E$ over $Z^o \ce N(\ell) - \ell$.
\item A \emph{framed} bundle $V^f$ on $Z_k \ce
      \Tot\bigl(\mathcal{O}_{\mathbb{P}^1}(-k)\bigr)$ is a pair
      consisting of a bundle $\pi_V \colon V \to Z_k$ together with a
      frame of $V$ over $Z^o_k$.
\item A \emph{framed} bundle $E^f$ on $X$ is a pair consisting of a
      bundle $\pi_E \colon E \to X$ together with a frame of $E$ over
      $N(x) -\{x\}$, where $N(x)$ is a small disk neighbourhood of $x$.
      We will always take $N(x) \ce \pi\bigl(N(\ell)\bigr)$.
\end{itemize}
\end{definition}

\begin{proposition}\label{glue}
An isomorphism class $[\bar E^f]$ of a framed bundle on $Z$ is
uniquely determined by a pair of isomorphism classes of framed bundles
$[E^f]$ on $X$ and $[V^f]$ on $Z_k$. We write $\bar{E}^f = (E^f, V^f)$.
\end{proposition}
\begin{proof}
One needs to observe that any reflexive sheaf on $X$ is
completely determined by the complement of the point $x$,
cf.\ \cite[Prop.\ 1.6]{Ha2}. Hence $E = (\pi_*\bar{E})^{\vee\vee}$ is
trivial on $N(x)$ and is completely determined by $E\vert_{X-\{x\}}$.
The rest of the proof is just a verification that the framings have
been conveniently defined.

The contraction map gives an isomorphism $i_1 \colon Z-\ell \to
X - \{x\}$. Based on \cite{Ga2} we may assume that there is an
isomorphism $i_2 \colon N(\ell) \to Z_k$. These induce
isomorphisms on the deleted neighbourhoods
\[ N(x) - \{x\} \xleftarrow{\ i_1 \ } Z^o \xrightarrow{\ i_2 \ } Z^o_k \text{ .} \]

By construction, $\bar{E} = E \sqcup_{(s_1, s_2) = (t_1, t_2)} V$ is
made by identifying the bundles as well as the sections over $Z^o$, so
that the bundles satisfy $\bar{E}\vert_{Z^o} = i_1^*
(E\vert_{{N(x)}-\{x\}}) = i_2^*(V\vert_{Z^o_k})$, and the framing
$(f_1, f_2)$ of $\bar{E}$ satisfies $(f_1, f_2)= (s_1 ,s_2) \circ i_1
= (t_1, t_2) \circ i_2$.

Let $\phi \colon E^f \to {E'}^f$ be an isomorphism such that $\phi
\circ (s_1, s_2) = (s'_1, s'_2)$, and let $\xi \colon V^f \to {V'}^f$
be an isomorphism such that $\xi \circ (t_1, t_2) = (t'_1, t'_2)$. We
have the following diagram of bundle maps:
\[\begin{CD}
  {\bar{E}^f\vert_{Z-\ell}} @>>> {E^f} @>{\phi}>> {E'^f} 
@<<< {\bar{E}'^f\vert_{Z-\ell}} \\
  @V{\pi_{\bar{E}}}VV @V{\pi_E}VV @VV{\pi_{E'}}V @VV{\pi_{\bar{E}'}}V \\
  {Z - \ell} @>>i_1{}> {X-\{x\}} @= {X-\{x\}} @<<{i_1}< {Z-\ell} 
\end{CD} \]
Hence, 
\begin{equation}\label{iso1}
  \bar{E}'^f\vert_{Z-\ell} = i_1^*(E'^f\vert_{X-\{x\}}) = i_1^* \circ \phi \,(E^f\vert_{X-\{x\}})
  = i_1^* \circ \phi \circ i_{1*} \,(\bar{E}^f\vert_{Z-\ell}) \text{ ,}
\end{equation}
showing that $i_1^* \circ \phi \circ i_{1*}$ is an isomorphism of
$\bar{E}$ and $\bar{E'}$ over $Z-\ell$ such that
\begin{equation}\label{iso2}
  \phi \circ (f_1, f_2) = \phi \circ (s_1, s_2) \circ i_1 = (s'_1, s'_2) \circ i_1 = (f'_1, f'_2) \text{ .}
\end{equation}

On the other hand we have a second diagram of bundle maps:
\[ \begin{CD}
  {\bar{E}^f\vert_{N(\ell)}} @>>> {V^f} @>{\xi}>> {V'^f} @<<< {\bar{E}'^f\vert_{N(\ell)}} \\
  @V{\pi_{\bar{E}}}VV @V{\pi_V}VV @VV{\pi_{V'}}V @VV{\pi_{\bar{V}'}}V \\
  {N(\ell)} @>>{i_2}> {Z_k} @= {Z_k} @<<{i_2}< {N(\ell)}
\end{CD} \]
 Therefore,
\begin{equation}\label{iso3}
  \bar{E}'^f\vert_{N(\ell)} = i_2^*({V}'^f\vert_{Z_k}) = i_2^* \circ \xi \,(V^f\vert_{Z_k})
  = i_2^* \circ \xi \circ i_{2*}\, (\bar{E}^f\vert_{N(\ell)}) \text{ ,}
\end{equation}
showing that $i_2^*\circ \xi \circ i_{2*}$ is an isomorphism of
$\bar{E}$ and $\bar{E'}$ over $N(\ell)$ such that
\begin{equation}\label{iso4}
  \xi \circ (f_1, f_2) = \xi \circ (t_1, t_2) \circ i_2 = (t'_1, t'_2) \circ i_2 = (f'_1, f'_2) \text{ .}
\end{equation}

These isomorphisms agree over the intersection $Z^o$. In fact, by
\eqref{iso1} and \eqref{iso3},
\begin{multline*}
  i_1^* \circ \phi \circ i_{1*} \,(\bar{E}^f\vert_{Z^o}) = i_1^* \circ \phi \,(E^f\vert_{N(x)})
  = i_1^*(E'^f\vert_{N(x)}) \\ = i_2^* (V'^f\vert_{Z_k-\{0\}}) = i_2^* \circ \xi \,(V^f\vert_{Z_k-\{0\}})
  = i_2^* \circ \xi \circ i_{2*} \,(\bar{E}^f\vert_{Z^o}) \text{ ,}
\end{multline*}
and moreover they also preserve the framings over the intersection,
since over $Z^o$ we have, by \eqref{iso2} and \eqref{iso4},
\[ \phi \circ (f_1, f_2) = (f'_1, f'_2) = \xi \circ (f_1, f_2) \text{ .} \]
By the gluing lemma this gives an isomorphism over the entire space
$\bar{X}$, and we get $\bar{E}' \simeq \bar{E}$.
\end{proof}




Note: Here we have only defined surgery for framed bundles on
surfaces. A similar definition of holomorphic surgery can be given in
much greater generality for decorated bundles on higher-dimensional
varieties; for instance, a broader sense of framing can be used by
fixing the isomorphism type of the bundles on a subvariety.

Using \eqref{KH} we can re-state Proposition \ref{glue} in terms of
instantons:

\begin{proposition}\label{patching}
If \ $\nabla$ and $\nabla'$ are instantons on $Z$, with $\nabla'$
obtained from $\nabla$ by local decay, then
\[ \gc(\nabla') = \gc(\nabla) - \lc(\nabla,\ell) \text{ .} \]
\end{proposition}
\begin{proof}
Just combine Proposition \ref{glue} and equality \eqref{blache}.
\end{proof}

\section{When is total decay near \texorpdfstring{$\bm\ell$}{l} possible?}

Consider the questions: Can every bundle decay totally around
$\ell$, that is, is every bundle related by holomorphic surgery to a
bundle that is trivial around $\ell$? Can decay by $1$ always happen,
that is can every charge $n$ instanton decay locally to charge $n-1$ ?
In the case $k = 1$ the answers are ``yes'', but for $k > 2$ we will
show that the answers to both questions are ``no''.

In this section we use the well-known concept of \emph{elementary
transformations} of Maru\-yama \cite{M}, which we now recall: Let $E$
be a vector bundle over an algebraic variety $W$. Choose a line bundle
$L$ over a Cartier divisor $D \subset W$ and a surjection $r \colon E
\to L$ induced by a surjection $\rho \colon E\vert_D \to L$. Set $E'
\ce \ker(r)$ and $L' \ce \ker(\rho)$. Since $D$ is a Cartier divisor,
$E'$ is a vector bundle on $W$. By definition $E'$ is called the
vector bundle obtained from $E$ by making the \emph{elementary
transformation} induced by $r$, denoted
\[ E' = \Elm_L(E) \text{ .} \]
The following diagram, called the \emph{display} of the 
elementary transformation, clarifies the situation:
\[ \begin{CD}
  & & {0} & & {0} \\
  @. @AAA @AAA \\
  {0} @>>> {L'} @>>> {E\vert_D} @>{\rho}>> {L} @>>> {0} \\
  @. @A{t}AA @AAA @| \\
  {0} @>>> {E'} @>>> {E} @>>{r}> {L} @>>> {0} \\
  @. @AAA @AAA \\
  & & {E(-D)} @= {E(-D)} \\
  @. @AAA @AAA \\
  & & {0} & & {0}
\end{CD} \]
Note that the elementary transformation does not change $E$ outside
the divisor, that is, $E\vert_{W-D} \cong E'\vert_{W-D}$.

\begin{proposition}\label{isomodk}
Let $E_1$ and $E_2$ be $\mathfrak{sl}(2, \mathbb{C})$-bundles over
$Z_k$ with splitting types $j_1$ and $j_2$, respectively. There exists
an isomorphism $E_1\vert_{Z_k^o} \cong E_2\vert_{Z_k^o}$ if and only
if $j_1 \equiv j_2 \mod k$. In particular, $E_1$ can decay totally
over $Z_k$ if and only if $j_1 \equiv 0 \mod k$.
\end{proposition}
\begin{proof}
We first claim that the bundle $\mathcal{O}_{\ell}(-k)$ is trivial on
$Z_k^o$. In fact, if $u=0$ is the equation of $\ell$, then $s(z, u) =
u$ determines a section of $\mathcal{O}_{\ell}(-k)$ that does not
vanish on $Z_k^o$.

If a bundle $E$ over $Z_k$ has splitting type $j$, then by definition,
$E\vert_{\ell} \cong \mathcal{O}_{\ell}(-j) \oplus
\mathcal{O}_{\ell}(j)$.  So there is a surjection $\rho \colon
E\vert_\ell \to \mathcal{O}_{\ell}(j)$. The bundle $E' =
\Elm_{\mathcal{O}_{\ell}(j)}(E)$ splits over $\ell$ as
$\mathcal{O}_{\ell}(-j) \oplus \mathcal{O}_{\ell}(j+k)$. Therefore we
can use the surjection $\rho \colon E'\vert_\ell \to
\mathcal{O}_{\ell}(j+k)$ to perform a second elementary
transformation, and we obtain bundle $E'' =
\Elm_{\mathcal{O}_{\ell}(j+k)}(E')$, which splits over $\ell$ as
$\mathcal{O}_{\ell}(-j) \oplus \mathcal{O}_{\ell}(j+2k)$ and has first
Chern class $2k$. Tensoring by $\mathcal{O}_{\ell}(-k)$ we get back to
an $\mathfrak{sl}(2, \mathbb{C})$-bundle with splitting type
$j+k$. Hence, the transformation
\[ \Phi(E) = \otimes \,\mathcal{O}_{\ell}(-k) \;\circ\; \Elm_{{\cal O}_{\ell}(j+k)}
   \;\circ\; \Elm_{{\cal O}_{\ell}(j)} (E) \]
increases the splitting type by $k$ while keeping the isomorphism type
of $E$ over $Z_k^o$. So we need only to analyse bundles with splitting
type $j < k$.

If $j = 0$, the bundle is globally trivial on $Z_k$. If $j \neq 0$,
then $E\vert_{Z^o_k}$ induces a non-zero element on the fundamental
group $\pi_1(Z_k^o) = \mathbb{Z}\bigl/k\mathbb{Z}$.
\end{proof}

One interesting consequence of Proposition \ref{isomodk} is that
instantons do not correspond to bundles whose splitting type does not
divide $k$. Consequently, using the results of the following two
sections, we will deduce:

\begin{corollary}
The self-intersection number of $\ell$ provides an obstruction to the
existence of instantons with certain prescribed numerical 
invariants.
\end{corollary}

In particular, it is not always possible for the local charge to decay by
one unless $k=1$ or $2$.

\begin{corollary}
The self-intersection number of $\ell$ poses an obstruction to
instanton decay.
\end{corollary}

\begin{example}
Here some examples, which will be proved below.
\begin{enumerate}
\item There is no nontrivial instanton with local charge $\leq k-2$
      over the space $Z_k$ when $k>2$. 
\item For $k \geq 2$, there exist $(k-2)$-dimensional families of
      (unframed) instantons with local charge $k-1$ over $Z_k$.
\end{enumerate}
\end{example}

\section{Existence of instantons}

In \cite{Ga2} it is shown that every holomorphic bundle on $Z_k$ is an
algebraic extension of line bundles. It then follows that any rank-$2$
bundle $E$ over $Z_k$ with $c_1(E) = 0$ is an extension of the form
\[ 0 \to \mathcal{O}(-j) \to E \to \mathcal{O}(j) \to 0 \text{ .} \]
Thus existence of moduli $\mathcal{M}_j(k)$ of bundles with any
splitting type $j$ over $Z_k$ is an immediate consequence of the fact
that $\Ext^1\bigl(\mathcal{O}_{Z_k}(j), \mathcal{O}_{Z_k}(-j)\bigr)
\neq \varnothing$. Moreover, in \cite[Theorem 4.2]{BG2} it is shown
that for $j > k$
\begin{equation}\label{localmoduli}\dim \mathcal{M}_j(k) = 2j - k - 2 \text{ .} \end{equation}
Note that $\dim\mathcal{M}_j(k)$ is not the dimension of
$\Ext^1\bigl(\mathcal{O}_{Z_k}(j), \mathcal{O}_{Z_k}(-j)\bigr)$ as a
vector space, since bundle isomorphisms impose several equivalences;
rather it is the dimension of the dense, open stratum of
$\mathcal{M}_j(k)$ seen as a variety. Now we analyse which of these
bundles  correspond to instantons. Firstly, we look at them from
the point of view of decay, and secondly from the point of view of
differential geometry.

The \emph{energy} of an instanton on $X$ is given by $\frac{1}{2g^2}
\int_X \norm{F}^2$. The minimum possible energy of the instanton is
bounded from below by the total charge of the instanton. To see this,
note the following inequalities:
\[ 0 \leq \int_X \norm{F \pm \hodge F}^2 = 2 \int_X \norm{F \wedge \hodge F \pm F \wedge  F} \]
Therefore,
\[ \frac{1}{2g^2} \int_X \norm{F}^2 = \frac{1}{2g^2} \int_X \norm{F \wedge \hodge F}
   \geq \frac{1}{2g^2} \left\vert \int_X F \wedge F \right\vert = \frac{8\pi^2|n|} {g^2} \text{ .}
\]
Consequently, the probability of finding an instanton with $c_2 = n$
is $\propto e^{-n}$. Arguing from the physical point of view, since
systems always tend to go to their lowest energy state, an instanton
with a high charge will prefer to decay to an instanton of a lower
charge unless there is some obstruction to its decay. It is reasonable
to expect that this behaviour 
holds locally as well. Thus it should be possible
to lower a local charge by a local transformation. Instanton bundles,
therefore, ought to allow for full local decay; combining with
Proposition \ref{isomodk} this implies that only bundles which are
trivial on $Z_k^o$ can correspond to instantons. Moreover, the
finite-energy condition for instantons, \textit{viz.}\ $\int_X
\norm{F}^2\leq\infty$, requires that $F \to 0$ at infinity, and
accordingly an instanton bundle $E$ on $Z_k$ should be trivial and
trivialised at infinity. This requirement in turn fixes boundary
conditions and guarantees that instantons have good gluing properties.

Mathematically, the correct way to decide which bundles correspond to
instantons is to go through the Kobayashi--Hitchin correspondence
(cf.\ \cite{LT}). A unitary, anti-self-dual connection $\nabla$ on a
smooth bundle $E$ decomposes as $\nabla = \partial + \bar\partial$,
where $\bar\partial$ is considered as a holomorphic structure on $E$;
and the Kobayashi--Hitchin correspondence claims that the map $\nabla
\mapsto \bar\partial$ is invertible. In the compact case, this
correspondence was proved by Donaldson \cite{D1} for projective
algebraic surfaces, by Uhlenbeck and Yau \cite{UY} for K\"{a}hler
surfaces, and by Buchdahl \cite{Bu} for surfaces with a Gauduchon
metric. In the non-compact case, this correspondence was proved by
Donaldson \cite{D2} for $\mathbb{C}^2$ and by King \cite{Kn} for
$\mathbb{C}^2$ blown up at the origin, which in this paper is denoted
by $Z_1$. In the former, instantons on $\mathbb{C}^2$ are identified
with instantons on $\mathbb{C}P^2$ framed at a line at infinity; and
in the latter framed instantons on $Z_1$ are identified with
instantons on the first Hirzebruch surface $\Sigma_1$ which are
trivial on the line at infinity. As in the non-compact cases of
$\mathbb{C}^2$ and $Z_1$ we identify framed instantons on $Z_k$ with
instantons on the $k^\mathrm{th}$ Hirzebruch surface $\Sigma_k \ce
\mathbb{P}\bigl(\mathcal{O}_{\mathbb{P}^1}(k) \oplus
\mathcal{O}_{\mathbb{P}^1}\bigr)$ trivialised on the line at
infinity.

LeBrun \cite{LB1} provided metrics over the spaces $Z_k$ that are
well suited for instanton problems. He showed that $Z_k$ admits a
complete, zero scalar curvature, asymptotically flat K\"{a}hler metric
$g$. In particular, this metric is anti-self-dual. Moreover, he showed
that up to multiplication by an overall constant, there is exactly one
such metric $g$ which is $SU(2)$-invariant. Using this asymptotically
flat metric, triviality at the line at infinity still seems a natural
condition to impose. In fact, consider the orbifold $\bar Z_k$
obtained from $Z_k$ by adding one point at infinity, or equivalently,
obtained from $\Sigma_k$ by contracting the line at infinity
$\ell_{\infty}$. Then it follows from \cite[p.~235]{LB2} that $\bar
Z_k$ is an ASD conformal orbifold compactification of $Z_k$. This
orbifold compactification of $Z_k$ has an orbifold twistor space $W$,
cf.\ \cite{LB3}. The complex structure on $Z_k$ yields a complex
surface in $W$, and adding the orbifold twistor line at infinity
compactifies this surface to the Hirzebruch surface $\Sigma_k$. Using
the Ward correspondence together with Uhlenbeck's removable
singularities theorem \cite{U}, an $L^2$-instanton on $Z_k$ gives rise
to a holomorphic bundle on $\Sigma_k$ that is trivial on
$\ell_{\infty}$. So, once again, from this second point of view we
arrive at the conclusion that instanton bundles on $Z_k$ should be the
ones that are trivial at infinity. We set the \textit{ad hoc}
definition of \emph{framed stability}, cf.\ Definition \ref{frame}.

\begin{definition}\label{stable}
A rank-$2$ bundle over $Z_k$ is called \emph{framed-stable} if it is
holomorphically trivial and framed on $Z^o_k$.
\end{definition}

This allows a statement the \emph{Kobayashi--Hitchin correspondence}
on $Z_k$:
 
\begin{proposition}\label{kobhitchin}
There exists a one-to-one correspondence between framed $SU(2)$-in\-stan\-tons
on $Z_k$ with local charge $n$ and framed-stable $\mathfrak{sl}(2, \mathbb{C})$-bundles
on $Z_k$ with $\chi^{\loc} = n$.
\end{proposition}

Schematically,
\begin{align}\label{localKH}
  \left\{ \parbox{4cm}{\centering $SU(2)$-instantons on\\$Z_k$ with local charge $n$} \right\}
  &\overset{\text{K.--H.}}{\Longleftrightarrow}
  \Bigl\{ \parbox{4cm}{\centering stable $SL(2)$-bundles\\on $Z_k$ with $\chi^{\loc} = n$} \Bigr\} \text{ .}
\end{align}

\begin{corollary}\label{inst}
An $\mathfrak{sl}(2, \mathbb{C})$-bundle over $Z_k$ represents an
instanton if and only if its splitting type is a multiple of $k$.
\end{corollary}
\begin{proof}
By definition an instanton bundle 
must be trivial at infinity, now apply
Proposition \ref{isomodk}.
\end{proof}

In particular, note that any bundle on $Z_k$ with splitting type $j
\not\equiv 0 \mod k$ does not correspond to an instanton; the
physical interpretation of such a bundle does not seem to be known.

\section{Gaps in local charges and moduli}

In this section we study gaps in local charges. We show that
not all numerically admissible values of $\chi$ occur for instanton 
bundles when $k > 2$. Recall that by Definition \ref{localcharge} the local
charge is the sum of the width and the height: $\chi^{\loc}(E) =
\mathbf{w}(E) + \mathbf{h}(E) = h^0\bigl(X; (\pi_*E)^{\vee\vee} \bigl/
\pi_* E \bigr) + h^0\bigl(X; R^1 \pi_* E \bigr)$.

\subsection{Direct computation of instanton widths}\label{sec.compwidth}

Results in this section depend on a number of ``direct
calculations''. We explain briefly how those are carried out and
provide an open computer implementation. We outline the computer
algorithm, following closely the ideas in \cite{GaS} and keeping to
minimal detail. The implementation of the algorithm can be obtained
from \url{http://www.maths.ed.ac.uk/~s0571100/Instanton/}.  Our
language of choice is \emph{Macaulay~2} for its native support of
high-level concepts of commutative algebra (such as modules,
generators, cokernels); though conceivably a different computer
algebra software may be used.

Let $E \to Z_k$ be a holomorphic rank-$2$ vector bundle over the
complex surface $Z_k = \Tot\bigl(\mathcal{O}(-k)\bigr)$ with $c_1(E) =
0$. The canonical coordinates on $Z_k = U \cup V$ are $U=\{z,u\}$ and
$V=\{\xi,v\}$ such that $\xi=z^{-1}$ and $v=z^k u$. A holomorphic
bundle $E$ on $Z_k$ is algebraic \cite{Ga2}. If the splitting type of
$E$ is $j$, then it can be expressed by a canonical transition function
\[ T = \begin{pmatrix} z^j & p \\ 0 & z^{-j} \end{pmatrix} \text{ ,} \]
where
\begin{equation}\label{extclass}
  p(z,u) = \sum_{r=1}^{\left\lfloor\frac{2j-2}{k}\right\rfloor} \sum_{s=kr-j+1}^{j-1} p_{rs} \; u^r z^s \text{ .}
\end{equation}

The computation of the instanton width is now equivalent to the
computation of the dimension of the cokernel of the natural evaluation
map $M \hookrightarrow M^{\vee\vee}$, where $M$ is a module that is
related to the space of holomorphic sections of $E$: Let $(a,b)$ be a
generic section of $E$ given over the $(z,u)$-chart by functions $a,b
\in \mathbb{C}[[z,u]]$. This implies that on the other chart, the
local section
\[ T\begin{pmatrix}a\\b\end{pmatrix} = \begin{pmatrix} z^j\,a + p\,b \\ z^{-j}\,b \end{pmatrix} \]
is holomorphic in $(z^{-1}, z^k\,u)$. Writing $a(z,u) = a_{rs} u^r
z^s$ and $b(z,u) = b_{rs} u^r z^s$, this means that for each fixed
index $r$ only a finite number of $a_{rs}$ and $b_{rs}$ can be
non-zero, and there are relations between the non-zero coefficients of
$a$ and $b$ (unless $p \equiv 0$). The space of sections of $E$ is
thus generated by terms $(u^r z^s, u^{r^\prime}z^{s^\prime})$, of
which most are multiples or linear combinations of a finite set of
true generators.

For our computation we need to consider generators and
relations after contracting the zero-section of $Z_k$ to a point,
i.e.\ on the direct image under the contracting map $\pi \colon Z_k
\to X_k$. Here $X_k$ is the singular surface (smooth only for $k=1$)
given in coordinates by $S = \bigl\{x_0, \dotsc, x_k\bigr\} \bigl/
(x_i x_{i+t} - x_{i+1} x_{i+t-1})$ with $0 \leq i \leq i+t \leq k+1$,
and the map $\pi$ is given by $x_i = z^i u$. The module $M$ is now the
space of sections of $E$ \emph{when the relations are expressed in
terms of the $x_i$ downstairs}, i.e.\ as an $S$-module.

\begin{example}
For the first two values of $k$ we have $Z_1 \to X_1 = \mathbb{C}^2 =
\bigl\{x_0,x_1\bigr\}$ and $Z_2 \to X_2 = \bigl\{x_0,x_1,x_2\bigr\} \bigl/
(x_0x_2-x_1^2)$, where $u \mapsto x_0$, $zu \mapsto  x_1$
and $z^2u \mapsto  x_2$.

The first two generators coming from $b_{00}$ and $b_{01}$ are
$\beta_0 = (0, 1)$ and $\beta_1 = (0, z)$, respectively, and they are
related over $S$ by
\begin{align*}
  \qquad&& x_1 \beta_0 &= x_0 \beta_1  && \text{(on $Z_1$ and $Z_2$), and} &&\qquad \\
  \qquad&& x_2 \beta_0 &= x_1 \beta_1  && \text{(on $Z_2$ only).} &&\qquad
\end{align*}
The concrete case $j=k$ and $p(z,u)=zu$ is worked out in the proof of
Theorem \ref{mini}.
\end{example}

\paragraph{Computer algorithm} The automatic computation of the instanton
width of a bundle $E$ with $c_1(E) = 0$, splitting type $j$ and
extension class $p$ can now proceed in several stages:
\begin{enumerate}
\item (Optional) The extension class $p$ may possibly be reduced to a
      smaller, cohomologous class $p'$ by truncating terms according
      to \eqref{extclass}, but care needs to be taken when $u \nmid
      p$. This step is only useful to optimise computation time, it is
      not necessary for the algorithm to work.
\item Define a generic section $(a,b)$ with $a(z,u) = \sum_{r,s}
      a_{rs} u^r z^s$ and $b(z,u) = \sum_{r,s}b_{rs} u^r z^s$. There
      exist bounds on $r$ above which all generators corresponding to
      the $a_rs$ and $b_rs$ terms are guaranteed to be multiples (over
      $S$) of the lower generators, so these are genuine
      polynomials. Moreover, $b$ can and must be chosen so that
      $z^{-j}\,b$ is holomorphic in $z^{-1}$ and $z^ku$.
\item Compute the first coordinate of the section on the second chart:
      $f = z^j\,a + p\,u$. Now for each term $z^s u^r$, whenever $s >
      kr$ the coefficient must vanish; this gives relations between
      the coefficients.
\item The module $M$ is now built up step by step by substituting the
      relations back into $a$ and $b$, and setting one coefficient to
      $1$ and all others to $0$ (call the result $a_1$, $b_1$ just for
      now), transforming $(a_1, b_1)$ into expressions over $S$ and
      adding the resulting vector as a generator of $M$. After 
      doing this for
      all coefficients, a presentation for $M$  as a module over $S$
       is obtained. (See
      \cite{GaS} for details of this construction, in particular how
      it deals with ``fake relations''.)
\item The computation of the cokernel of $\operatorname{ev} \colon M
      \hookrightarrow M^{\vee\vee}$ relies on \cite[Lemma 2.1]{GaS}
      and can be done very easily in \emph{Macaulay~2}. The instanton
      width of $E$ is the dimension of $\coker(\operatorname{ev})$ as
      a $\mathbb{C}$-vector space.
\end{enumerate}

\subsection{Computation of instanton heights}

We will use the following formula for the height, which is proved in
\cite{BG2}:

\begin{theorem}[{\cite[2.6]{BG2}}]
Let $E$ be a non-split bundle represented in canonical form by
$(j,p)$, and let $m > 0$ be the smallest exponent of $u$ appearing in
$p$. With $\mu = \min\bigl(m, \lfloor\frac{j-2}{k}\rfloor\bigr)$, we
have
\begin{equation}\label{eq.height}
  l(R^1\pi_*E) = \mu \left(j - 1  - k\;\frac{\mu - 1}{2}\right) \text{  ,}
\end{equation}
and equality holds if $p$ is holomorphic on $Z_k$.
\end{theorem}

\subsection{Results}

\begin{lemma}\label{jlessk}
Let $E_j$ be a rank-$2$ bundle over $Z_k$ with $c_1 = 0$ and splitting
type $j < k$. Then
\begin{equation}\label{chiloc} \chi^{\loc} (E_j) = j - 1 \text{ .} \end{equation}
\end{lemma}
\begin{proof}
By \cite[Theorem 3.3]{Ga2} it follows that if $j<k$ then $E_j \cong
\mathcal{O}_{Z_k}(j) \oplus \mathcal{O}_{Z_k}(-j)$. By definition,
$\chi^{\loc}(E_j) = \mathbf{w}(E_j) + \mathbf{h}(E_j)$. Direct
computation (see \cite{BG2}) then shows that $\mathbf{w}(E_j) = 0$ and
$\mathbf{h}(E_j) = j - 1$.
\end{proof}

\begin{proposition}\label{prop.mins}
Let $E$ be a rank-$2$ reflexive sheaf on $Z_k$ with splitting type
$\geq k$, then $\chi\bigl(x, E\bigr) \geq k - 1$.
\end{proposition}
\begin{proof}
By semi-continuity of $\chi\bigl(x, E\bigr)$ on the splitting type,
the lowest value of $\chi\bigl(x, E\bigr)$ must occur for splitting
type $k$. By definition,
\[ \chi\bigl(x, E\bigr) \geq h^0\bigl(X; R^1\pi_*E\bigr) \text{ ,} \]
and now apply the formula \eqref{eq.height}.
\end{proof}

\begin{theorem}\label{mini}
The minimal local charge of a nontrivial $SU(2)$-instanton on $Z_k$ is
$\chi^\mathrm{min}_k = k-1$. The local moduli space of (unframed)
instantons on $Z_k$ having fixed local charge $\chi^\mathrm{min}_k$
has dimension $k-2$.
\end{theorem}
\begin{proof}
By Corollary \ref{inst}, a nontrivial instanton bundle over $Z_k$ has
splitting type $kn$ for some $n \in \mathbb{Z}$, $n>0$. Hence, the
smallest nontrivial splitting type is exactly $k$,
and the generic such instanton corresponds to an
element of $\Ext^1_{Z_k}\bigl(\mathcal{O}(k), \mathcal{O}(-k)\bigr)$
which is nontrivial on the first formal neighbourhood.

The dimension of the local moduli space with fixed
$\chi^\mathrm{min}_k$ is obtained from formula \eqref{localmoduli}
setting $j = k$.  We compute $\chi$ for the bundle $E$ corresponding
to the extension class $[zu] \in \Ext^1\bigl(\mathcal{O}(k),
\mathcal{O}(-k)\bigr)$.

By Equation \eqref{eq.height}, $\mathbf{h}(E) = k - 1$.


To compute $\mathbf{w}(E) = h^0\bigl(X, (\pi_*E)^{\vee\vee} \bigl/
\pi_*E\bigr)$, we use the method described in Section
\ref{sec.compwidth}: Let $Q$ be the skyscraper sheaf defined by the
exact sequence
\[ 0 \longrightarrow \pi_*E \longrightarrow \bigl(\pi_*E\bigr)^{\vee\vee} \longrightarrow Q \longrightarrow 0 \text{ .} \]
Then $\mathbf{w}(E)$ equals the dimension of $Q_x^{\wedge}$ as a
$k_x$-module. So, we need to study the map
$\pi_*E^{\wedge}_x \to (\pi_*E_x^{\wedge})^{\vee\vee}$
and compute the dimension of the cokernel, i.e.\ we need to compute the module structure on $M \ce
\pi_*E_x^\wedge$ and study the natural map
$M \hookrightarrow M^{\vee\vee}$. By the Theorem on Formal Functions
(see \cite[p.~277]{Ha1}),
\[ M \cong \varprojlim_{n} H^0\bigl(\ell_n; E\rvert_{\ell_n}\bigr) \text{ ,} \]
where $\ell_n$ is the $n^\text{th}$ infinitesimal neighbourhood of
$\ell$. Since the extension class has degree $1$ in $u$, then for $n
\geq 1$,
\[ H^0\bigl(\ell_n; E\rvert_{\ell_n} \bigr) \cong H^0\bigl(\ell_1; E\rvert_{\ell_1} \bigr) \text{ .} \]
Therefore, the inverse limit stabilises at $1$, giving $M \cong
H^0\bigl(\ell_1; E\rvert_{\ell_1} \bigr)$.

To compute the generators of $M$, we write the transition matrix for
$E$ explicitly. We set $Z_k = U \cup V$, where $U \cong \mathbb{C}^2
\cong V$, with change of coordinates $U \ni (z,u) \mapsto (z^{-1} ,
z^ku)\in V $ on $U \cap V \cong \mathbb{C}-\{0\} \times
\mathbb{C}$. Then in these coordinate charts, $E$ is given
by transition matrix
\[ T = \begin{pmatrix} z^k & zu \\ 0 & z^{-k} \end{pmatrix} \text{ .} \]
Set $\alpha = \bigl(\begin{smallmatrix} u \\ 0 \end{smallmatrix}
\bigr)$ and $\beta_i = \bigl(\begin{smallmatrix} 0 \\ z^i
\end{smallmatrix} \bigr)$ for $i = 0,\dotsc, k-1$, $\beta_k = 
\bigl(\begin{smallmatrix} -zu \\ z^k \end{smallmatrix} \bigr)$. Then a
presentation for $M$ is given by $M=\bigl< \alpha, \beta_0, \dotsc,
\beta_k\bigr> \bigl/ R$, where $R$ is the set of relations $\beta_i
x_0 - \beta_{i-1} x_1 = 0$ for $i = 1, \dotsc, k-1$ and $\beta_k x_0 -
\beta_{ik-1} x_1 - \alpha x_1= 0$. Standard computations (which can be
performed either by hand, or with a computer algebra program) then
show that the evaluation map $\rho \colon M \to M^{\vee\vee}$ is
surjective. Hence $\mathbf{w}(E) = 0 $.

Summing up, $\chi^\mathrm{min}_k = \chi^{\loc}\bigl(E\bigr) =
\mathbf{w}(E) + \mathbf{h}(E) = 0 +(k-1)$.
\end{proof}

\begin{remark}[Gaps in  local instanton charges]
The non-existence of instantons with certain local charges on the
spaces $Z_k$ when $k>2$ is in stark contrast with what happens in the
case $k=1$. In fact, by \cite[Theorem 0.2]{BG1}, for every
non-negative integer $n$ there exist instantons on $Z_1$ with local
charge $n$. More precisely, by \cite[Theorem 0.2]{BG1}, for every pair
of integers $(w, h)$ satisfying $j - 1 - e \leq w \leq j(j-1)/2 - j\,
e$ and $1 \leq h \leq j(j+1)/2$ with $j \geq 0$ and $e \geq 0$ or
$-1$, there exists a rank-$2$ vector bundle $E$ on $Z_1$ with
splitting type $(j, -j + e)$ having  $\mathbf{w}(E) =
w$ and $\mathbf{h}(E) = h$. Hence, there are no gaps in local charges
for instantons over $Z_1$.
\end{remark}

\bigskip\bigskip\vfill

\noindent Elizabeth Gasparim, Thomas K\"{o}ppe \\
School of Mathematics, The University of Edinburgh \\
James Clerk Maxwell Building, The King's Buildings \\
Mayfield Road, Edinburgh, UK, EH9 3JZ \\
E-mail: \url{Elizabeth.Gasparim@ed.ac.uk} \\
E-mail: \url{t.koeppe@ed.ac.uk}

\bigskip

\noindent Pushan Majumdar\\
Institut f\"{u}r Theoretische Physik \\
Westf\"alische Wilhelms-Universit\"{a}t M\"{u}nster \\
48149 M\"{u}nster, Germany \\
E-mail: \url{pushan@uni-muenster.de}


\begin{thebibliography}{MMMM}
\addcontentsline{toc}{section}{References}

\bibitem[BG1]{BG1} Ballico, E.\ and Gasparim, E., \emph{Vector bundles
on a neighborhood of a curve in a surface and elementary
transformations}, Forum Math. \textbf{15} no. 1, 115--122 (2003).

\bibitem[BGK]{BG2} Ballico, E.\  Gasparim, E.\  and  K\"{o}ppe, T.,
\emph{Vector bundles near a negative curve: moduli and local 
Euler characteristic.} Preprint.

\bibitem[Bl]{Bl} Blache, R., \emph{Chern classes and
Hirzebruch--Riemann--Roch theorem for coherent sheaves on
complex-projective orbifolds with isolated singularities}, Math.\ Z.
\textbf{222} no. 1, 7--57 (1996).

\bibitem[Bu]{Bu} Buchdahl, N.\ P., \emph{Hermitian--Einstein connections
and stable vector bundles over compact algebraic surfaces}, Math. Ann. 
\textbf{280}, 625--648 (1988).

\bibitem[BHMM]{BHMM} Boyer, C.\ P., Hurtubise, J.\ C., Mann, B.\ M. and
Milgram, R.\ J., \emph{The topology of instanton moduli spaces. I. The 
Atiyah--Jones conjecture.} Ann.\ of Math.\ (2) \textbf{137} no. 3, 561--609 (1993).

\bibitem[D1]{D1} Donaldson, S.\ K., \emph{Instantons and geometric
invariant theory}, Comm.\ Math.\ Phys. \textbf{93}, 453--460 (1984).

\bibitem[D2]{D2} Donaldson, S.\ K., \emph{Anti-self-dual connections
over complex algebraic surfaces and stable vector bundles},
Proc.\ Lond.\ Math.\ Soc.\ (3) \textbf{50}, 1--26 (1985).

\bibitem[Ga1]{Ga1} Gasparim, E., \emph{The Atiyah--Jones conjecture 
for rational surfaces.} To appear in Advances Math. (2008).

\bibitem[Ga2]{Ga2} Gasparim, E., \emph{Holomorphic bundles on 
${\mathcal O}(-k)$ are algebraic}, Comm.\ Algebra \textbf{25} no. 10, 3001--3009 (1997).

\bibitem[GaS]{GaS} Gasparim, E.\ and Swanson, I., \emph{Computing instanton 
numbers of curve singularities}, J. Symbolic Comput. \textbf{40}
no. 2, 965--978 (2005)

\bibitem[GO]{GO} Gasparim, E.\ and Ontaneda, P., \emph{Three applications of instanton numbers},
Comm.\ Math.\ Phys. \textbf{270} no.1, 1--12 (2007).

\bibitem[Ha1]{Ha1} Hartshorne, R., \emph{Algebraic geometry}. Graduate Texts in Mathematics \textbf{52}. Springer-Verlag, New York-Heidelberg (1977).

\bibitem[Ha2]{Ha2} Hartshorne, R., \emph{Stable reflexive sheaves}, Math.\ Ann. \textbf{254} no. 2, 121--176 (1980).

\bibitem[HM]{HM} Hurtubise, J.\ and Milgram, R.\ J., \emph{The
Atiyah--Jones conjecture for ruled surfaces}, J.\ Reine Angew.\
Math. \textbf{466}, 111--143 (1995).

\bibitem[Kn]{Kn} King, A., \emph{Instantons and holomorphic bundles on 
the blown-up plane}, D.\ Phil.\ Thesis, Worcester College, Oxford (1998).

\bibitem[LB1]{LB1} LeBrun, C., \emph{ Counter-examples to the
generalized positive action conjecture}, Commun.\ Math.\ Phys. \textbf{118}, 591--596 (1988).

\bibitem[LB2]{LB2} LeBrun, C., \emph{Explicit self-dual metrics on $\mathbb{C}P^2\#\cdots\#\;\mathbb{C}P^2$},
J.\ Differential Geom. \textbf{34} no. 1, 223--253 (1991).

\bibitem[LB3]{LB3} LeBrun, C., \emph{Twistors, K\"{a}hler manifolds, and bimeromorphic geometry. I.}
J.\ Amer.\ Math.\ Soc. \textbf{5} no. 2, 289--316  (1992).

\bibitem[LT]{LT} L\"{u}bke, M.\ and Teleman, A., \emph{The
Kobayashi--Hitchin correspondence.} World Scientific Publishing Co.,
Inc., River Edge, NJ (1997).

\bibitem[M]{M} Maruyama, M., \emph{Elementary transformations in the
theory of algebraic vector bundles}, Algebraic geometry (La
R\'{a}bida, 1981), 241--266, Lecture Notes in Math. \textbf{961},
Springer, Berlin, 1982.

\bibitem[U]{U} Uhlenbeck, K.\ \emph{Removable singularities in
Yang-Mills fields}, Comm.\ Math.\ Phys. \textbf{83} no. 1, 11--29
(1982).

\bibitem[UY]{UY} Uhlenbeck, K.\ and Yau, S.\ T., \emph{On the
existence of Hermitian--Yang--Mills connections in stable vector
bundles}, Comm.\ Pure Appl.\ Math. \textbf{39}, suppl. S257--S293
(1986).

\end{thebibliography}
\end{document}